\documentclass[11pt,a4paper,reqno]{amsart}

\usepackage[utf8]{inputenc}
\usepackage[T1]{fontenc}
\usepackage[british]{babel}
\usepackage{amsmath,amsthm,amssymb}
\usepackage{hyperref,float,caption,comment}
\hypersetup{
    colorlinks,
    linkcolor={red!60!black},
    citecolor={green!60!black},
    urlcolor={blue!60!black},
}
\usepackage{geometry}
\usepackage{enumitem,thm-restate}
\usepackage[abbrev, msc-links, nobysame]{amsrefs}
\usepackage{tikz,esvect}
\usetikzlibrary{arrows}
\tikzset{
e+ /.tip = {[sep=0pt]|[sep=3pt]_},
e+4 /.tip = {[sep=0pt]|[sep=5pt]_},
e+5 /.tip = {[sep=0pt]|[sep=5pt]_},
e+6 /.tip = {[sep=0pt]|[sep=6pt]_},
e+2 /.tip = {[sep=0pt]|[sep=2pt]_}
}
\tikzset{arc/.style = {->,> = stealth'}}
\tikzset{
	dot/.style = {circle, fill, minimum size=#1,
		inner sep=0pt, outer sep=0pt},
	dot/.default = 6pt
}
\usetikzlibrary{angles,quotes}
\usepackage[nameinlink, capitalise, noabbrev, sort]{cleveref}

\def\N{\mathbb N}
\def\Z{\mathbb Z}

\def\C{\mathcal C}

\def\A{\mathcal A}
\def\col{{\rm Col}}
\def\x{{\bf x}}
\def\y{{\bf y}}
\def\a{{\bf a}}
\def\b{{\bf b}}
\def\c{{\bf c}}
\def\f{{\bf f}}
\def\1{{\bf 1}}
\def\0{{\bf 0}}
\DeclareMathOperator{\cycw}{cycw}
\DeclareMathOperator{\dtw}{dtw}

\DeclareMathOperator{\cp}{cp}

\newtheorem{theorem}{Theorem}[section]
\newtheorem{lemma}[theorem]{Lemma}
\newtheorem{corollary}[theorem]{Corollary}
\newtheorem{proposition}[theorem]{Proposition}

\newtheorem{question}[theorem]{Question}
\theoremstyle{definition}
\newtheorem{definition}[theorem]{Definition}

\newtheorem{remark}[theorem]{Remark}

\Crefname{question}{Question}{Questions}

\title[Hitting cycles through prescribed vertices or edges]{Hitting cycles through prescribed vertices or edges}

\author[N.\,Bowler]{Nathan Bowler}
\author[E.\,Ghorbani]{Ebrahim Ghorbani}
\address{Ebrahim Ghorbani:
Hamburg University of Technology, Institute for Algorithms and Complexity, Hamburg, Germany and Universit{\"a}t Hamburg, Department of Mathematics, Bundesstra{\ss}e~55 (Geomatikum), 20146~Hamburg, Germany}
\email{ebrahim.ghorbani@tuhh.de}
\author[F.\,Gut]{Florian Gut}
\author[R.\,W.\,Jacobs]{Raphael W. Jacobs}
\author[F.\,Reich]{Florian Reich}
\address{Nathan Bowler, Florian Gut, Raphael W. Jacobs, Florian Reich: Universit{\"a}t Hamburg, Department of Mathematics, Bundesstra{\ss}e~55 (Geomatikum), 20146~Hamburg, Germany}
\email{\{nathan.bowler,  florian.gut, raphael.jacobs, florian.reich\}@uni-hamburg.de}

\begin{document}

\begin{abstract}
    We prove that for every set~$S$ of vertices of a directed graph~$D$, the maximum number of vertices in~$S$ contained in a collection of vertex-disjoint cycles in~$D$ is at least the minimum size of a set of vertices that hits all cycles containing a vertex of~$S$.
    As a consequence, the directed tree-width of a directed graph is linearly bounded in its cycle-width, which improves the previously known quadratic upper bound.

    We further show that the corresponding statement in bidirected graphs is true and that its edge-variant holds in both undirected and directed graphs, but fails in bidirected graphs.
    The vertex-version in undirected graphs remains an open problem.
\end{abstract}

\maketitle 

\section{Introduction}

How many vertices of a graph~$G$ does it take to hit all cycles in~$G$?
The \emph{Feedback Vertex Set Problem} which looks for a minimal such set of vertices of~$G$ is among Karp's classical~NP-complete problems~\cite{karp1972reducibilitycombinatorialproblems}.
The celebrated Erd\H{o}s-P\'{o}sa Theorem asserts a qualitative duality between feedback vertex sets and collections of vertex-disjoint cycles:

\begin{theorem}[\cite{erdos1965independentcircuitscontained}] \label{thm:ErdosPosa}
	There exists~$c \in \N$ such that, for each~$k \in \N$, every undirected graph $G$ contains either~$k$ vertex-disjoint cycles or a set~$X$ of at most~$c\, k \log k$ vertices that hits every cycle of $G$.
\end{theorem}
\noindent
In the context of directed graphs the duality between feedback vertex set and vertex-disjoint cycles is known as Younger's conjecture and was proven by Reed, Robertson, Seymour and Thomas:

\begin{theorem}[\cite{reed1996packingdirectedcircuits}] \label{thm:YoungersConjecture}
	For each~$k \in \N$ there exists an integer $d \in \N$ such that every directed graph $D$ contains either~$k$ vertex-disjoint cycles or a set~$X$ of at most~$d$ vertices that hits every cycle of $D$.
\end{theorem}
\noindent
The duality between feedback edge sets and edge-disjoint cycles in undirected and directed graphs holds true in the same way (by applying a simple auxiliary construction~\cite{pontecorvi2012disjointcyclesintersecting} or the directed line graph~\cite{heinlein2018directed}, respectively).

More generally, one can consider \emph{$S$-cycles}, that is, cycles hitting a prescribed set~$S$ of vertices.
Wollan and Pontecorvi extended \cref{thm:ErdosPosa} to $S$-cycles:
\begin{theorem}[\cite{pontecorvi2012disjointcyclesintersecting}] \label{thm:ErdosPosaVertices}
	There exists~$c \in \N$ such that, for each~$k \in \N$, undirected graph~$G$ and set~$S \subseteq V(G)$, the graph~$G$ contains either~$k$ vertex-disjoint $S$-cycles or a set~$X$ of at most~$c\, k \log k$ vertices hitting every~$S$-cycle in~$G$.
\end{theorem}
\noindent
Kakimura and Kawarabayashi extended \cref{thm:YoungersConjecture} fractionally to $S$-cycles~\cite{kakimura2012packing}, and Kawarabayashi, Kr{\'a}l', Kr{\v{c}}{\'a}l and Kreutzer later did so half-integrally~\cite{kawarabayashi2013packing}.
Both these results as well as \cref{thm:ErdosPosaVertices} transfer to edge-disjoint cycles through prescribed sets of edges
(again by applying a simple auxiliary construction~\cite{pontecorvi2012disjointcyclesintersecting} or the directed line graph, respectively).
\medskip

Given a set~$S$ of vertices of an undirected graph, every set of~$k$ vertex-disjoint $S$-cycles clearly witnesses that a minimum feedback vertex set for the graph's~$S$-cycles has size at least~$k$.
The converse of this does not hold, as \cref{thm:ErdosPosaVertices} is sharp:
there exist graphs and vertex sets~$S$ with at most~$k$ vertex-disjoint~$S$-cycles, where every feedback vertex set for the~$S$-cycles has size at least~$\Omega(k \log k)$.
But instead we can ask for a set of vertex-disjoint~$S$-cycles which meets $S$ as often as the minimum size of a feedback vertex set for the~$S$-cycles.

To state this formally, we write, for a collection $\C$ of cycles in an (un)directed graph, $V(\C)$ for the set of all vertices that are contained in some cycle in~$\C$ and~$E(\C)$ for the analogous set of edges.
\begin{question}\label{quest:vertex}
    Let $G$ be an undirected graph and let $S \subseteq V(G)$.
    Is it true that
	\begin{align*}
		&\max \{|S \cap V(\mathcal{C})| \colon \mathcal{C} \text{ is a collection of vertex-disjoint cycles of } G \} \\
		&\ge\min \{|X| \colon X \subseteq V(G) \text{ hits all $S$-cycles in~$G$}\}?
	\end{align*}
\end{question}
\noindent The same question arises for vertex-disjoint cycles in directed and bidirected graphs as well as for edge-disjoint $F$-cycles meeting a prescribed set of edges~$F$.

In the present paper, we study these various versions of~\cref{quest:vertex} and obtain the following results:

\begin{table}[h!]
  \centering
  \begin{tabular}{c|c|c|c}
     & \textbf{undirected graph} & \textbf{directed graph} & \textbf{bidirected graph} 
     \\
    \hline
    \textbf{vertex-version} & \textit{open} & \textit{true}, see \cref{thm:CycleDirGraphVertex} & \textit{true}, see \cref{mainthm} \\
    \hline
    \textbf{edge-version} & \textit{true}, see \cref{thm:EdgeVersionUndir} & \textit{true}, see \cref{main:DirectedGraphs} & \textit{false}, see \cref{prop:BidirectedFalse} \\
  \end{tabular}
  \label{table:results}
\end{table}

To our surprise, we have not been able to deduce the undirected statement from its directed analogue.
The naive approach suggests that we should replace every undirected edge by two parallel directed edges with distinct directions, but this creates many new cycles (of length~$2$) and as such gives no direct correspondence between the cycles in the undirected graph and the cycles in the corresponding directed graph.
Similar methods which replace the two parallel directed edges by a more involved `gadget' seem to fail for the same reason.
Therefore, \cref{quest:vertex} remains an open conjecture.

The proofs of the various versions of~\cref{quest:vertex} mostly rely on algebraic methods.
For the edge-version in undirected graphs, \cref{thm:EdgeVersionUndir}, the proof uses the cycle space over~$\Z_2$ as a key tool.
The proof of the edge-version in directed graphs, \cref{main:DirectedGraphs}, is based on the strong duality theorem for linear programming.
The vertex-version in directed graphs, \cref{thm:CycleDirGraphVertex}, follows from the respective statement in bidirected graphs, \cref{mainthm}, whose proof also relies on the strong duality theorem for linear programming. \\

As an application of the vertex-version in directed graphs, \cref{thm:CycleDirGraphVertex}, we improve a bound between two different width-measures on directed graphs, directed tree-width and cycle-width.
\emph{Directed tree-width} is the well-established directed analogue of tree-width in undirected graphs~\cite{johnson2001directed}, exhibiting similar features such as a directed grid theorem using butterfly minors~\cite{kawarabayashi2015directed} and a respective cops \& robbers-game~\cite{johnson2001directed}.
The structure of directed graphs is known to be closely connected to the structure of bipartite graphs with perfect matchings; see e.g.~\cite{wiederrecht2020thesis} and especially its Section~3.2 for an overview.
One of the central open questions on the structure of graphs with perfect matchings is Norine's matching grid conjecture~\cite{norine2005matching}, which asserts an analogue of the (directed) grid theorem in terms of perfect matching width and matching minors. 
The notion of \emph{cycle-width} was introduced by Hatzel, Rabinovich and Wiederrecht~\cite{hatzel2019cyclewidth} to bridge the gap between directed tree-width and perfect matching width.
Roughly speaking, they showed that cycle-width is equivalent to both directed tree-width and perfect matching width, which allowed them to apply the directed grid theorem to resolve Norine's matching grid conjecture for bipartite graphs.

Cycle-width is a branch decomposition-type width-measure of directed graphs where the order of a cut measures how often a family of vertex-disjoint cycles can meet the edges in the cut (see~\cref{def:CycleWidth} for the formal definition).
Together with the cops \& robbers-game corresponding to directed tree-width~\cite{johnson2001directed}, \cref{thm:CycleDirGraphVertex} allows us to show the following upper bound on directed tree-width in terms of cycle-width:
\begin{restatable}{corollary}{corDirTWCycleW}  \label{cor:DirTWCycleW}
	For every directed graph~$D$, we have~$\dtw(D) \le 9 \cycw(D) - 2$.
\end{restatable}
\noindent This improves the previously known quadratic upper bound for directed graphs by Giannopoulou, Kreutzer and Wiederrecht~\cite{giannopoulou2021excluding} and, together with their lower bound, implies that directed tree-width and cycle-width are linearly bounded in each other.
Moreover, our proof of~\cref{cor:DirTWCycleW} is significantly simpler than their technically involved proof.
We note that their proof contains a variant of \cref{quest:vertex} in directed graphs with weaker bounds, which inspired us to the present work.

Bidirected graphs can be obtained from undirected graphs by assigning to each endvertex of an edge one of two signs, $+$ or a $-$; directed graphs can then be viewed as those bidirected graphs in which each edge has sign $+$ at one endvertex and sign $-$ at the other (see~\cref{subsec:VertexBidirected} for the formal definition).
While directed graphs characterise bipartite graphs with perfect matchings, bidirected graphs play the same role for general graphs with perfect matchings~\cite{wiederrecht2020thesis}.
This is why, in his doctoral thesis~\cite{wiederrecht2020thesis}, Wiederrecht suggested the investigation of structural properties of bidirected graphs as a route towards the full resolution of Norine's matching grid conjecture~\cite{norine2005matching}.

Following this route, we also study~\cref{quest:vertex} in bidirected graphs, and we show with~\cref{mainthm} that the vertex-version in bidirected graphs is true.
As for directed graphs, \cref{mainthm} implies a variant of \cref{cor:DirTWCycleW} for bidirected graphs, i.e. bidirected tree-width is linearly bounded by cycle-width (\cref{cor:BidiTWCycleW}).
In contrast, the edge-version in bidirected graphs does not hold (\cref{prop:BidirectedFalse}).
Our counterexample builds on the fact that Menger's theorem does not transfer to bidirected graphs without additional assumptions~\cite{bowler2023mengertheorembidirected}.

\medskip

This paper is structured as follows.
We first state and prove our results on the various versions of~\cref{quest:vertex} in~\cref{sec:Proofs}.
In~\cref{sec:DirTWandCycleW} we then discuss the concepts of (bi)directed tree-width and cycle-width and prove~\cref{cor:DirTWCycleW} and its bidirected analogue, \cref{cor:BidiTWCycleW}.

For all graph-theoretic terms and notations, we follow~\cite{DiestelBook2024}.

\section{Proofs of the main results} \label{sec:Proofs}

\noindent In this section, we state and prove our main results.

\subsection{The edge-version in undirected graphs} \label{subsec:ProofUndirEdge}

\begin{restatable}{theorem}{mainTheoUndir} \label{thm:EdgeVersionUndir}
	Let~$G$ be an undirected graph, and let~$F \subseteq E(G)$ be a set of edges. Then
	\begin{align*}
		&\max \{|F \cap E(\mathcal{C})| \colon \mathcal{C} \text{ is a collection of edge-disjoint cycles of } G \} \\
		&\ge\min \{|Y| \colon Y \subseteq E(G) \text{ hits all $F$-cycles in~$G$} \}.
	\end{align*}
\end{restatable}

\begin{proof}
    Let $\A$ be the set of all~$F$-cycles in~$G$, and let~$M$ be the incidence matrix of~$E(\A)$ versus~$\mathcal{A}$:
    The rows and columns of~$M$ correspond to the edges in~$E(\A)$ and the cycles in~$\A$, respectively.
    For~$e \in E(\A)$ and~$C \in \A$, the entry at position~$(e, C)$ then is $1$ if~$e \in C$ and $0$ if~$e \notin C$.
    
    The column space of~$M$ over the field $\Z_2$, denoted as $\col_{\Z_2}(M)$, is a subspace\footnote{Indeed, $\col_{\Z_2}(M)$ can be proved to be equal to the cycle space of the induced subgraph of~$G$ on~$E(\A)$. However, we do not need this property in our argument.} of the cycle space of the induced subgraph of~$G$ on~$E(\A)$.
    This implies that every vector in $\col_{\Z_2}(M)$ represents a collection of edge-disjoint cycles.
    Indeed, the mod-2 sum of some columns of~$M$, corresponding for example to cycles~$C_1, \ldots, C_\ell \in \A$, gives the vector representing the symmetric difference~$E(C_1) \triangle \cdots \triangle  E(C_\ell)$, which induces a collection of edge-disjoint cycles in~$G$.

    Now let~$t$ be the value on the right-hand side of the inequality in the statement.
    By the above, it suffices to find a vector in~$\col_{\Z_2}(M)$ whose restriction to~$F$ has at least~$t$~many~$1$'s, which is what we prove below.

For every $E'\subseteq E(\A)$ let us denote by $M_{E'}$ the submatrix of~$M$ consisting of the rows corresponding to edges in $E'$.
    We claim that the rank of~$M_F$ over~$\Z_2$ is at least~$t$.
    Otherwise, there exists~$F' \subseteq F$ of size~$< t$ such that every row of~$M_F$ is the mod-2 sum of the rows of~$M_{F'}$.
    This implies that all cycles in~$\A$ meet~$F'$:
    Every column of~$M_F$ is non-zero by the definition of~$M$, and since the rows of~$M_{F'}$ generate every other row of~$M_F$, every column of~$M_F$ must have at least one entry equal to $1$ in $M_{F'}$-part.    In other words, $F'$ hits all cycles in~$\A$ and has size~$< t$, contradicting the minimality of~$t$.
    
    Since~$M_F$ has rank at least~$t$ over~$\Z_2$, we may assume that the matrix $M$ has the following form in which the $t\times t$ submatrix $N_{t\times t}$ is non-singular over $\Z_2$:
    \begin{equation*}
        M =~
        \begin{array}{|cccccc|}
            \hline
            \hspace{-.17cm}
                \begin{array}{c|}
                    \hspace{-.12cm} N_{t\times t} \hspace{-.12cm} \\
                    \hline 
                \end{array}
                &&&&& \vspace{-.2cm} \\ && M_F&&& \\	&&&&&\vspace{-.3cm}\\
            \hline
            &&&&&\\&&&&&\\
            \hline
        \end{array}
        \,.
    \end{equation*}
    But since~$N_{t\times t}$ is non-singular, its columns can generate every vector in~$\Z_2^t$ and in particular the all-$1$ vector of length~$t$.
    Thus, there is a vector in~$\col_{\Z_2}(M)$ whose first~$t$ entries are~$1$, as desired.
\end{proof}

\subsection{The edge-version in directed graphs} \label{subsec:ProofDirEdge}

\begin{restatable}{theorem}{mainTheo}\label{main:DirectedGraphs}\label{thm:CycleDirGraph}
    Let~$D$ be a directed graph, and let~$F \subseteq E(D)$ be a set of edges. Then 
\begin{equation}\label{eq:edge-version-di}
	\begin{array}{rl}
		&\max \{|F \cap E(\C)| \colon \mathcal{C} \text{ is a collection of edge-disjoint cycles of~} D \} \vspace{.1cm}\\
		&\ge\min \{|Y| \colon Y \subseteq E(D) \text{ hits all $F$-cycles in~$D$}\}.
	\end{array}
 \end{equation}
\end{restatable}

\begin{proof}
The proof relies on the strong duality theorem of linear programming.
To this end, we construct a primal linear program whose optimal value equals
the left-hand side of \eqref{eq:edge-version-di}, and the optimal value of the
dual is an upper bound for the right-hand side of \eqref{eq:edge-version-di}.

    Let $n$ be the number of vertices of $D$ and let $m$ be the number of edges of $D$.
    Further, let $N_{n\times m}$ be the incidence matrix of $D$, and set
    $$A:= \begin{bmatrix} I_m \\ N \end{bmatrix},$$ 
    where $I_m$ is the identity matrix of size $m$.
    Let $\f$ be the characteristic vector of $F$ as a subset of $E(D)$, and let
    $$\b:=\begin{bmatrix}\1_m\\\0_n \end{bmatrix},$$
    where $\1_m$ is the all-1 vector of length $m$.
    Consider the primal linear program
    \begin{align*}
        \text{ maximise } &\f^\intercal \x \\
        \text{ subject to } &A\x \leq \b \\
        &\x \geq 0.
    \end{align*}
    Recall that an integer-valued matrix is totally unimodular if every non-singular square submatrix has determinant $+1$ or $-1$.
    The incidence matrix of every directed graph, and thus $N$, is known to be totally unimodular~\cite{biggs1996}*{Proposition~5.3}.
It turns out that $A$ is also totally unimodular.
To see this, consider an arbitrary square submatrix $R$ of $A$.
If $R$ is itself a submatrix of $N$ or of $I_m$, or if $R$ contains a row or a
column consisting entirely of zeros, then $\det R \in \{0,\pm1\}$ and there is
nothing to prove.
Otherwise, by expanding the determinant of $R$ along the rows corresponding to
the $I_m$ block, we obtain $\det R = \pm \det R'$,
where $R'$ is a square submatrix of $N$. Since $N$ is totally unimodular, it follows that $\det R \in \{0,\pm1\}$, and hence $A$ is totally unimodular.
    Therefore, both the primal linear program and its dual have integral optimal solutions (see e.g.~\cite{schrijver1998}*{Corollary 7.1g}).
    
    Let $\x^*$ be an integral optimal solution of the primal linear program.
    The constraints demand~$\x^*$ to be a $\{0,1\}$-vector in the null-space of $N$.
    Consequently, the set of edges $\{ e \in E(D) :\x^*(e) = 1 \}$
    induces a directed subgraph of $D$ in which, at every vertex, the number of incoming edges equals the number of outgoing edges. 
    Therefore, $\x^*$ represents the set of edges of some edge-disjoint cycles in $D$, where its optimality implies that it has the maximum intersection with $F$.
    Thus, $\f^\intercal \x^*$ gives the left-hand side of \eqref{eq:edge-version-di}.
    
    Now we consider the dual program
    \begin{align*}
    	\text{ minimise } &\b^\intercal \y\\
    	\text{ subject to } &A^\intercal \y \geq \f \\
    	&\y \geq 0.
    \end{align*}
    Let
    $$\y^*:=\begin{bmatrix}\y_1\\\y_2\end{bmatrix}$$
    be an integral optimal solution of the dual linear program, where $\y_1$ has length $m$ and $\y_2$ has length $n$. 
    Let $Y$ be the set of edges in $D$ whose corresponding components in $\mathbf{y}_1$ are non-zero. 
    As $\y_1$ is a non-negative integral vector, we have $|Y|\le\1^\intercal\y_1=\b^\intercal \y^*$. 
    By the strong duality theorem of linear programming (see e.g.~\cite{schrijver1998}*{Corollary~19.2b}), $\b^\intercal \y^*=\f^\intercal \x^*$.
    So we will be done by showing that $Y$ hits all $F$-cycles.
     
    Let $C$ be a cycle with the characteristic vector $\c$ (as a subset of $E(D)$).    
    We can write  $A^\intercal \y^* \geq \f$ as $\y_1+N^\intercal \y_2\ge \f$.
    Multiplying both sides by $\c^\intercal$ from the left, and observing that $\c^\intercal N^\intercal=\0$, we get $\c^\intercal\y_1\ge\c^\intercal \f$.
    If $C$ hits $F$, then $\c^\intercal \f=|E(C)\cap F|>0$, implying $\c^\intercal\y_1>0$, which means that $C$ hits $Y$ as well.
\end{proof}

\subsection{The vertex-version in bidirected graphs} \label{subsec:VertexBidirected}

Let us briefly recall the relevant definitions around bidirected graphs; for a complete formal introduction, we refer the reader to~\cite{bowler2023mengertheorembidirected}.

A \emph{bidirected graph}~$B = (G, \sigma)$ consists of a graph~$G$, which is undirected and may contain parallel edges but no loops, and a \emph{signing}~$\sigma$ which assigns to each half-edge~$(u, e)$ of~$G$, where~$e \in E(G)$ and~$u \in e$, one of the signs~$+$ and~$-$; we say that~\emph{$e$ has sign~$\sigma(u,e)$ at~$u$}.
All terms such as vertices, edges, adjacent, incident, edge- and vertex-disjoint transfer directly from~$G$ to~$B$.
A \emph{path} in~$B$ is a path in~$G$ such that for every internal vertex~$v$ of the path, the two incident edges have distinct signs at~$v$; in other words, the sign switches at each internal vertex of the path.
\emph{Cycles} in~$B$ are then the analogous sign-switching version of cycles in~$G$.

\begin{theorem}\label{mainthm}
    Let $B$ be a bidirected graph and let $S \subseteq V(B)$ be a set of vertices.
    Then
    \begin{align*}
    	& \max \{|V(\mathcal{C}) \cap S|: \mathcal{C} \text{ a collection of vertex-disjoint cycles} \} \\
    	&\geq  \min \{|X|: X \subseteq V(D) \text{ hits all $S$-cycles}\}.
    \end{align*}
\end{theorem}

The proof of \cref{mainthm} applies linear programming in a similar way to the proof of~\cref{main:DirectedGraphs}, i.e. solutions to the primal program correspond to edge-disjoint cycles.
For this we transfer the vertex-problem into an edge-problem:
Let $B'$ be the bidirected graph obtained from $B$ by replacing each vertex $v \in V(B)$ with two vertices $v^+$, $v^-$ and making $v^+$ incident with all edges incident with $v$ with a plus sign and making $v^-$ incident with all edges incident with $v$ with a minus sign.
Furthermore, connect $v^-$ and $v^+$ by an edge with minus sign at $v^+$ and plus sign at $v^-$.

We observe that:
\begin{itemize}
    \item[(i)] there is a one-to-one correspondence between the cycles of $B$ and those of $B'$;
    \item[(ii)] two cycles in $B$ are vertex-disjoint if and only if the corresponding cycles in $B'$ are vertex-disjoint;
  \item[(iii)]  every two edge-disjoint cycles in $B'$ are also vertex-disjoint.  
\end{itemize}
\noindent
Although the edge-version in bidirected graphs is not true in general (see~\cref{subsec:BiDirEdge}), for such a bidirected graph $B'$ as constructed above and the set of edges induced by $S$ it is equivalent to the vertex-version in bidirected graphs and we will show it to hold.

As a further ingredient to the proof of \cref{mainthm}, we need the following concepts around incidence matrices of bidirected graphs.
The \emph{incidence matrix} of a bidirected graph $(G, \sigma)$ is defined based on the signs of the `half-edges', where the sign of the entries $+1$ or $-1$ corresponding to a vertex $v$ and an edge $e$ is equal to $\sigma(v,e)$.

\begin{remark}\label{remark}
    For the bidirected graph $B'$, every $\{0,1\}$-vector $\x$ in the null-space of its incidence matrix corresponds to the edge set of some union of vertex-disjoint cycles in $B'$ (this is because, in the subgraph induced by $\{ e \in E(B') :\x^*(e) = 1 \}$, the edge between $v^-$ and $v^+$ can be used at most once, which implies that
    each vertex in this subgraph has exactly two incident edges).
    This property fails for general bidirected graphs (see \cref{fig:forRemark} for an example), which explains why the following proof will not work for the general edge-version problem.
\end{remark}

\begin{figure} [H]
    \centering
    \begin{tikzpicture}
        \node[dot] (a0) at (0,0) {};
        \node[dot] (c2) at (-1.5,-.75) {};
        \node[dot] (c1) at (-1.5,.75) {};
        \node[dot] (b1) at (1.5,.75) {};
        \node[dot] (b2) at (1.5,-.75) {};
        \draw[thick] (a0) to (b1);
        \draw[thick] (a0) to (b2);
        \draw[-e+6, thick] (c1) to (a0);
        \draw[-e+6, thick] (c2) to (a0);
        \draw[-e+, thick] (c2) to (c1);
        \draw[-e+, thick] (c1) to (c2);
        \draw[-e+, thick] (b2) to (b1);
        \draw[-e+, thick] (b1) to (b2);
    \end{tikzpicture}
    \caption{A bidirected graph for which the columns in its incidence matrix sum up to zero because each vertex has an equal number of positive and negative incident edges, yet it has no bidirected cycle.}
    \label{fig:forRemark}
\end{figure}
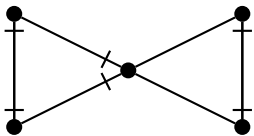

While the total unimodularity of incidence matrices of directed graphs relies on the characteristic property that each column has exactly one $+1$ and one $-1$ entry, this characteristic does not necessarily hold for bidirected graphs, whose incidence matrices may thus not be totally unimodular. 
Despite this, the following more general notions allow us to extend the integrality of optimal solutions of the linear programs considered for directed graphs to bidirected graphs.

A rational matrix $M$ is called {\em $k$-regular}, for a positive integer $k$, if for every non-singular submatrix $R$ of $M$, $kR^{-1}$ is integral.
If $M$ is $k$-regular, then the matrix $\frac{1}{k} M$ is $1$-regular.
Indeed, every non-singular submatrix of $\frac{1}{k}M$ is of the form $\frac{1}{k}R$,
where $R$ is a non-singular submatrix of $M$, and its inverse is given by 
$\left(\frac{1}{k}R\right)^{-1} = kR^{-1}$ which is integral.

\begin{lemma}[\cite{APPA2004k-regular}*{Theorem~23}]\label{lem:BiDirIncidence2Regular}
    The incidence matrix of a bidirected graph is $2$-regular.
\end{lemma}

Further, a polyhedron $P=\{\x: \x\ge\0, M\x\le\a\}$ is called \emph{integral} if
$\max\{\c^\intercal \x: \x \in P\}$ is attained by an integral vector for each $\c$ for which the maximum is finite. 
\begin{lemma}[\cite{APPA2004k-regular}*{Theorem 16}]\label{lem:integral-Ax<a}
    Let $M_{n \times m}$ be a rational matrix.
    Then $M$ is $1$-regular if and only if for every integral vector $\a$ of length $n$ the polyhedron $\{\x: \x\ge\0, M\x\le\a\}$ is integral.
\end{lemma}

\begin{proof}[Proof of \cref{mainthm}]
    Let $n$ be the number of vertices of $B'$ and let $m$ be the number of edges of $B'$.
    Further, let $F\subseteq E(B')$ be the set of all edges constructed from vertices in $S \subseteq V(B)$.
    It suffices to establish the following:	
    \begin{equation}\label{eq:edge-version-bidi}
    	\begin{array}{rl}
            &\max \{|\mathcal{C} \cap F|: \mathcal{C} \text{ a set of edge-disjoint cycles in $B'$} \} \vspace{.1cm} \\
            &\geq  \min \{|Y|: Y \subseteq E(B') \text{ hits all $F$-cycles in $B'$} \}.
    	\end{array}
    \end{equation}
    	
    Let $\f$ be the characteristic vector of $F$ as a subset of $E(B')$, and let
    \begin{equation*}
        \b:=
        \begin{bmatrix}
            \1_m\\
            \0_n\\
            \0_n
        \end{bmatrix}.
    \end{equation*}
    Also let $M$ be the incidence matrix of $B'$, and 
    \begin{equation*}
        A:=
        \begin{bmatrix}
            I_m \\
           \frac12M \\
            -\frac12M
        \end{bmatrix}.
    \end{equation*}
    Consider the primal linear program
    \begin{align*}
        \text{ maximise } &\f^\intercal \x \\
        \text{ subject to } &A\x \leq \b \\
        &\x \geq 0.
    \end{align*}
    By~\cref{lem:BiDirIncidence2Regular}, $M$ is $2$-regular and hence $\frac{1}{2}M$ is $1$-regular.
    It follows that
    $\begin{bmatrix}
        \frac12M \\
      -\frac12M
        \end{bmatrix}$
    is also $1$-regular because every non-singular submatrix of it can be viewed as a submatrix of $\frac{1}{2}M$ with some rows possibly negated.
    Stacking an identity matrix on a $1$-regular matrix again yields a $1$-regular matrix (see \cite{APPA2004k-regular}*{Lemma~7}), so $A$ is $1$-regular as well.

    According to \cref{lem:integral-Ax<a}, the polyhedron $P := \{\x: \x \geq \0, A\x \leq \b\}$ is integral.
    Moreover, $\max\{\f^\intercal \x: \x \in P\}$ is finite, since $A\x \le b$ and $\x \ge 0$ require all entries of $\x$ to be contained in the interval~$[0,1]$.
    Thus, the primal linear program has an integral optimal solution $\x^*$.
    The constraints demand $\x^*$ to be a $\{0,1\}$-vector in the null-space of $M$.
    Therefore, by \cref{remark}, $\x^*$ represents the set of edges of some union of edge-disjoint cycles in $B'$, where its optimality implies that it has the maximum intersection with $F$.
    Thus, $\f^\intercal \x^*$ gives the left-hand side of~\eqref{eq:edge-version-bidi}.
    
    Now we consider the dual program
    \begin{align*}
    	\text{ minimise } &\b^\intercal \y\\
    	\text{ subject to } &A^\intercal \y \geq \f \\
    	&\y \geq 0,
    \end{align*}
    and show that it has an integral optimal solution.
    By the strong duality theorem of linear programming (see e.g.~\cite{schrijver1998}*{Corollary~19.2b}), the dual program has an optimal solution~$\y*$ with~$\b^\intercal \y^*=\f^\intercal \x^*$.
    The matrix $-A^\intercal$ is $1$-regular since $A$ is $1$-regular.
    Thus, by \cref{lem:integral-Ax<a}, the polyhedron $\{\x: \x\ge\0, -A^\intercal \x\le - \f \} = \{\y: \y\ge\0, A^\intercal \y\le \f \}$ is integral.
    Since maximising $-\b^\intercal \y$ is equal to minimising $\b^\intercal \y$, we may hence assume $\y^*$ to be an integral optimal solution
    \begin{equation*}
        \y^*=
        \begin{bmatrix}
            \y_1\\
            \y_2\\
            \y_3
        \end{bmatrix}
    \end{equation*}
    for the dual program, where $\y_1$ has length $m$ and both $\y_2$ and $\y_3$ have length $n$. 
    Let $Y$ be the set of edges in $B'$ whose corresponding components in $\mathbf{y}_1$ are non-zero.
    As $\y_1$ is a non-negative integral vector, we have $|Y|\le\1^\intercal\y_1=\b^\intercal \y^*$. 
    Since~$\b^\intercal \y^*=\f^\intercal \x^*$, we will thus be done by showing that $Y$ hits all $F$-cycles.
     
    Let $C$ be an arbitrary cycle in $B'$ and write $\c$ for the characteristic vector as a subset of $E(B')$.    
    We can write $A^\intercal \y^* \geq \f$ as $\y_1+\frac12 M^\intercal (\y_2-\y_3)\ge \f$.
    Multiplying both sides by $\c^\intercal$ from the left, and observing that $\c^\intercal M^\intercal=\0$, we get $c^\intercal\y_1\ge\c^\intercal \f$.
    If $C$ hits $F$, then $\c^\intercal \f=|E(C)\cap F|>0$, implying $\c^\intercal\y_1>0$, which means  $C$ hits $Y$ as well.
    
    So we have proved that $|Y|$ is at most the left-hand side of \eqref{eq:edge-version-bidi}, and at least the right-hand side of \eqref{eq:edge-version-bidi}, completing the proof.
\end{proof}

\subsection{The vertex-version in directed graphs} \label{subsec:ProofDirVertex}
Every directed graph $D$ forms a bidirected graph in the following way:
Let $G$ be the underlying undirected graph of $D$ and let $\sigma$ be the map sending a half edge $(u,e)$ to $+$ if $u$ is the head of $e$ and to $-$ if $u$ is the tail of $e$.
Then $(G, \sigma)$ is a bidirected graph whose cycles correspond one-to-one to the cycles of $D$.

We can deduce from \cref{mainthm}:

 \begin{theorem} \label{thm:CycleDirGraphVertex}
    Let $D$ be a directed graph, and let~$S \subseteq V(D)$ be a set of vertices. Then
	\begin{align*}
		&\max \{|S \cap V(\mathcal{C})| \colon \mathcal{C} \text{ a collection of vertex-disjoint cycles in } D \} \\
		&\ge\min \{|X| \colon X \subseteq V(D) \text{ hits all $S$-cycles in~$D$}\}.
	\end{align*}
\end{theorem}

\subsection{A counterexample to the edge-version in bidirected graphs} \label{subsec:BiDirEdge}
The following proposition now asserts that the edge-version of~\cref{quest:vertex} does not hold in bidirected graphs: 
\begin{proposition} \label{prop:BidirectedFalse}
	For every~$k \in \N$ there is a bidirected graph $B_k$ and a set~$F_k$ of edges of~$B_k$ such that
	\begin{enumerate}[label=(\arabic*)]
		\item \label{item:OnlyOneElementOfFk} for every collection $\mathcal{C}$ of edge-disjoint cycles in~$B_k$, we have $|F_k \cap E(\mathcal{C})| \leq 1$, and
		\item \label{item:NoSetOfSizeK} for every edge set $X_k$ of size at most $k$, $B_k - X_k$ contains a cycle meeting~$F_k$.
	\end{enumerate}
\end{proposition}

\noindent For the proof of~\cref{prop:BidirectedFalse}, we use the following result:

\begin{theorem}[\cite{bowler2023mengertheorembidirected}*{Theorem~3.3}] \label{thm:counterexample_edge_menger}
    For each number~$k \in \N$, there exists a bidirected graph~$A_k$ and distinct vertices~$x$ and~$y$, such that there are no two edge-disjoint~$x$--$y$~paths and for each subset~$S \subseteq E(A_k)$ of size~$k$ there exists an~$x$--$y$~path in~$A_k - S$.
\end{theorem}

\begin{proof}[Proof of~\cref{prop:BidirectedFalse}]
	Let~$A_k$, $x$ and~$y$ be as in~\cref{thm:counterexample_edge_menger}.
	We may assume without loss of generality that all edges incident to~$x$ in~$A_k$ have sign~$+$ at~$x$ and that all edges incident to~$y$ in~$A_k$ have sign~$-$ at~$y$.
	Now let~$B_k$ be the bidirected graph obtained from~$A_k$ by identifying~$x$ and~$y$, and let~$F_k$ be the set of edges with sign~$-$ at~$x = y$, i.e.\ the set of edges incident to~$y$ in~$A_k$.
    For an illustration of $B_k$, refer to \cref{fig:counterexample_bidirected_cp_hs_bidi}.
    In this figure, the blue vertex at the top-left corner represents the vertex resulting from the identification of $x$ and $y$ and $F_k$ consists of the dotted edges which were initially incident to $y$.
    We claim that~$B_k$ and~$F_k$ are as desired. 
	
	For~\labelcref{item:OnlyOneElementOfFk}, let~$C$ be an arbitrary cycle in~$B_k$ containing an edge of~$F_k$ (like e.g.\ the one in red in~\cref{fig:counterexample_bidirected_cp_hs_bidi}).
	The cycle~$C$ contains exactly one edge with sign~$-$ at~$x = y$ and exactly one edge with sign~$+$ at~$x = y$; in particular, $C$ contains exactly one edge of~$F_k$.
	Furthermore, the edge set~$E(C)$ induces an~$x$--$y$~path in~$A_k$.
	Since there are no two edge-disjoint~$x$--$y$~paths in~$A_k$, every collection $\mathcal{C}$ of edge-disjoint cycles has at most one element that hits $x = y$.
    In particular, $|F_k \cap E(\mathcal{C})| \leq 1$ holds for every collection~$\mathcal{C}$ of edge-disjoint cycles in~$B_k$.
	
	For~\labelcref{item:NoSetOfSizeK}, let~$X_k \subseteq E(B_k) = E(A_k)$ with~$|X_k| \leq k$ be arbitrary.
	There is an~$x$--$y$~path~$P$ in~$A_k - X_k$ by~\cref{thm:counterexample_edge_menger}.
	By the construction of~$B_k$, the edge set~$E(P)$ induces a cycle in~$B_k - X_k$, and this cycle contains an edge of~$F_k$.
\end{proof}

	\begin{figure}
	\begin{tikzpicture}[scale=.85]
		\node[dot,label={[xshift=-0.9cm, yshift=0.1cm,color={blue}]$x = y$},fill={blue}] (a) at (-.5,7) {};
		\node[dot] (b2) at (1.5,1.5) {};
		\node[dot] (b3) at (2.33,1.5) {};
		\node[dot] (c2) at (1.5,3) {};
		\node[dot] (c3) at (2.25,3) {};
		\node[dot] (c5) at (3.75,3) {};
		\node[dot] (c6) at (4.5,3) {};
		\node[dot] (d2) at (1.5,4.5) {};
		\node[dot] (d3) at (2.25,4.5) {};
		\node[dot] (d5) at (3.75,4.5) {};
		\node[dot] (d6) at (4.5,4.5) {};
		\node[dot] (d8) at (6.25,4.5) {};
		\node[dot] (d9) at (7,4.5) {};
		\node[dot] (e2) at (1.5,6) {};
		\node[dot] (e3) at (2.25,6) {};
		\node[dot] (e5) at (3.75,6) {};
		\node[dot] (e6) at (4.5,6) {};
		\node[dot] (e8) at (6.25,6) {};
		\node[dot] (e9) at (7,6) {};
		\node[dot] (e11) at (8.5,6) {};
		\node[dot] (e12) at (9.25,6) {};
		\draw [-e+5, thick] (e2) ..controls+(-1cm,0) and +(1.2cm,.2)  .. (a);
	    \draw [-e+6, thick] (d2) ..controls+(-1cm,0) and +(1.2cm,-1)  .. (a);
		\draw [-e+6, thick] (b2) ..controls+(-1cm,-.2) and +(-1.2cm,-2)  .. (a);
	\draw [-e+, thick] (b3) ..controls+(2cm:-5) and +(left:2cm)  .. (a);
	\draw [-e+, thick] (b3) ..controls+(2cm:-5) and +(left:2cm)  .. (a);
	\draw [dashed, thick] (e2) ..controls+(0,1.5) and +(.5cm,1)  .. (a);
	\draw [dashed, thick] (e5) ..controls+(0,2) and +(.5cm,1.5)  .. (a);
	\draw [dashed, thick] (e11) ..controls+(0,2.5) and +(.5cm,2.5)  .. (a);
	\draw [dashed, thick] (e12) ..controls+(3,2) and +(.5cm,3.5)  .. (a);	
		\foreach \x in {b,d,e} 
		{\draw[-e+2, thick] (\x2) to (\x3);
			\draw[-e+2, thick] (\x3) to (\x2);}
		\foreach \x in {d,e} 
		{\draw[-e+2, thick] (\x5) to (\x6);
			\draw[-e+2, thick] (\x6) to (\x5); }
		\draw[-e+2, thick] (e11) to (e12);
		\draw[-e+2, thick] (e12) to (e11);
		\foreach \x/\y in {b3/c6,d9/e12} 	
		\draw [thick] (\x) ..controls+(right:1.9cm) and +(down:1.9cm)  .. (\y);	
		\foreach \x/\y in {2/3,5/6}
		\draw[thick] (e\y) to (d\x); 	
		\foreach \x/\y in {2/3,5/6}
		\draw[thick] (d\y) to (c\x); 	
		\draw[thick] (c3) to (b2); 	
		\foreach \x in {d,e}
		\draw[thick] (\x3) to (\x5); 
		\foreach \x in {d,e}
		\draw[thick] (\x6) to (\x8); 
		\draw[thick] (e9) to (e11);
        \draw [-e+, thick, color={red}] (c2) ..controls+(-1cm,0) and +(.5cm,-3.5)  .. (a);
        \draw[-e+2, thick, color={red}] (c2) to (c3);
        \draw[-e+2, thick, color={red}] (c3) to (c2);
        \draw[thick, color={red}] (c3) to (c5);
        \draw[-e+2, thick, color={red}] (c5) to (c6);
        \draw[-e+2, thick, color={red}] (c6) to (c5);
        \draw [thick, color={red}] (c6) ..controls+(right:1.9cm) and +(down:1.9cm)  .. (d9);
        \draw[-e+2, thick, color={red}] (d9) to (d8);
        \draw[-e+2, thick, color={red}] (d8) to (d9);
        \draw[thick, color={red}] (d8) to (e9);
        \draw[-e+2, thick, color={red}] (e9) to (e8);
        \draw[-e+2, thick, color={red}] (e8) to (e9);
        \draw [dashed, thick, color={red}] (e8) ..controls+(0,2) and +(.5cm,2)  .. (a);
	\end{tikzpicture}  
 \vspace{-2.5cm}
	    \caption{A counterexample to the edge-version in bidirected graphs. More precisely, a bidirected graph in which there are no two edge-disjoint cycles containing a dotted edge. For every edge set $X$ of size $\leq 2$ there exists a cycle in $B - X$ containing a dotted edge.}
	    \label{fig:counterexample_bidirected_cp_hs_bidi}
	\end{figure}
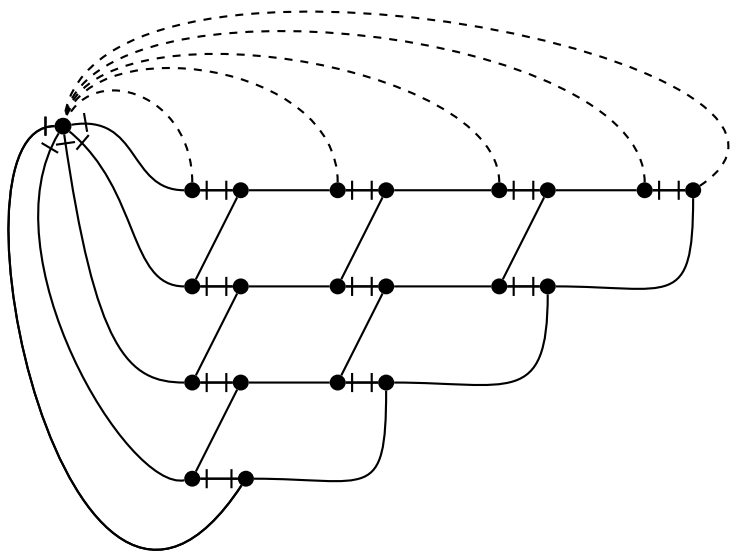

\section{An Application: (Bi)directed Tree-Width and Cycle-Width} \label{sec:DirTWandCycleW}

Cycle-width was introduced by Hatzel, Rabinovich and Wiederrecht~\cite{hatzel2019cyclewidth} as a tool to make the directed grid theorem of Kawarabayashi and Kreutzer~\cite{kawarabayashi2015directed} accessible in the realm of bipartite graphs with perfect matchings.
They used it to resolve Norine's matching grid conjecture~\cite{norine2005matching} for the case of bipartite graphs.
The key features of cycle-width are that it is equivalent to both a respective matching-theoretic width-measure, called perfect matching width, and directed tree-width.
Here, we are concerned with this later equivalence between cycle-width and directed tree-width, which we show to be linear.

\corDirTWCycleW*

With the same proof, we will also show the corresponding statement for bidirected graphs, namely that bidirected tree-width is linearly bounded by cycle-width.
We hope this to be useful towards the full resolution of Norine's matching grid conjecture~\cite{norine2005matching} along the route suggested by Wiederrecht~\cite{wiederrecht2020thesis}.
Our statement holds for bidirected graphs $B=(G, \sigma)$ that are \emph{circular}, that is, every edge of $B$ is contained in a cycle and $G$ is connected.
This, however, is no restriction in the context of strong connectivity in bidirected graphs (see~\cite{wiederrecht2020thesis}*{text before~Lemma~9.2.18}).

\begin{corollary} \label{cor:BidiTWCycleW}
    For every circular bidirected graph $B$, we have $\mathrm{btw}(B) \leq 9 \cycw(B) + 2$.
\end{corollary}

Towards the proof of~\cref{cor:DirTWCycleW} and \cref{cor:BidiTWCycleW}, we begin with the formal definition of the cycle-width of a (bi)directed graph.
\begin{definition}[\cite{wiederrecht2020thesis}*{Definition 5.3.11}] \label{def:CycleWidth}
	A \emph{cycle-decomposition} of a (bi)directed graph~$D$ is a pair~$(T, \varphi)$ of a cubic tree~$T$ and a bijection~$\varphi$ between~$V(D)$ and the set~$L(T)$ of leaves of~$T$.
	Each edge~$t_1 t_2$ of~$T$ then \emph{induces} a cut~$E_G(V_1, V_2)$ of~$D$ where~$V_i = \varphi(L(T_i))$ and~$T_i$ is the subtree of~$T - t_1 t_2$ containing~$t_i$ for~$i = 1, 2$.
	
	The~\emph{cycle porosity} of a cut~$F$ of~$D$ is defined as
	\begin{equation*}
		\cp(F) := \max_{\substack{\text{$\C$ collection of} \\ \text{pairwise vertex-disjoint} \\ \text{cycles in~$D$}}} |F \cap E(\C)|.
	\end{equation*}
	The \emph{width} of a cycle-decomposition~$(T, \varphi)$ of~$D$ is then defined as the maximum cycle porosity of a cut of~$D$ induced by an edge of~$T$.
	The \emph{cycle-width}~$\cycw(D)$ of~$D$ is the minimum width of a cycle-decomposition of~$D$.
\end{definition}

Our proof of~\cref{cor:DirTWCycleW} and \cref{cor:BidiTWCycleW} proceeds analogously to the one of Giannopoulou, Kreutzer and Wiederrecht in their quadratic upper bound~\cite{giannopoulou2021excluding}*{Proposition 3.3}; we include the full formal proof here for completeness.
The proof does not use the definition of (bi)directed tree-width itself, but a sufficient condition for bounded (bi)directed tree-width via a suitable `cops \& robbers-game' on (bi)directed graphs, \cref{thm:CopsAndRobbers,thm:CopsAndRobbersBidi} below.
We therefore use (bi)directed tree-width as a black box here and refer the reader to~\cite{wiederrecht2020thesis}*{Definition~9.2.17} for the definition.

\begin{definition}[\cite{johnson2001directed},\cite{wiederrecht2020thesis}*{Definition~9.2.11}]
    The \emph{cops \& robbers-game} is played between $k$ cops and a robber on a (bi)directed graph~$D$ as follows: For the start the cops take an initial position~$C_0 \subseteq V(D)$ and the robber chooses her initial position~$R_0$, a strong component of~$D - C_0$.

    After the~$i$-th round, let~$C_i \subseteq V(D)$ be the position of the cops and let~$R_i$ be the strong component of~$D - C_i$ in which the robber is situated.
    Since there are $k$ cops, we stipulate that $|C_i| \leq k$.
    The cops then announce their new position~$C_{i+1} \subseteq V(D)$ and the robber afterwards chooses as her position a strong component~$R_{i+1}$ of~$D - C_{i+1}$ such that~$R_i \cup R_{i+1} \subseteq R$ for some strong component~$R$ of~$D - (C_i \cap C_{i+1})$.
    If no such~$R_{i+1}$ exists, then the cops \emph{catch} the robber after~$i+1$ rounds.

    If $k$ cops can ensure that they catch the robber after a finite number of rounds independently of the robber's moves, then the cops have a \emph{winning strategy}, and we say that \emph{$k$ cops can catch the robber on~$B$}.
\end{definition}

 \begin{theorem}[\cite{johnson2001directed}] \label{thm:CopsAndRobbers}
     For every directed graph~$D$ and every~$k \in \N$, we have~$\dtw(D) \le 3k - 2$ if~$k$ cops can catch the robber on~$D$.
 \end{theorem}

\begin{theorem}[\cite{wiederrecht2020thesis}*{Lemma 9.2.13, Theorem 9.2.20}] \label{thm:CopsAndRobbersBidi}
    For every circular bidirected graph~$B$ and every~$k \in \N$, we have~$\mathrm{btw}(B) \le 3k + 2$ if~$k$ cops can catch the robber on~$B$.
\end{theorem}

\begin{proof}[Proof of~\cref{cor:DirTWCycleW} and \cref{cor:BidiTWCycleW}]
    Before we start with the proof, we remark that the only part of the proof where we distinguish between directed and circular bidirected graphs is the very end where we apply~\cref{thm:CopsAndRobbers,thm:CopsAndRobbersBidi}, respectively, and hence get a small additive distinction between the respective bounds.
    
    Let $D$ be either a directed graph or a circular bidirected graph.
    Let~$(T, \varphi)$ be a cycle-decomposition of~$D$ of width~$k := \cycw(D)$.
    So for every edge~$e \in E(T)$, the induced cut $C_e$ has cycle-porosity at most~$k$.

    Let $D_e$ be the (bi)directed graph obtained by adding a subdivision vertex to each edge of $C_e$ and let $S_e$ be the set of all these subdivision vertices.
    Since the cycle-porosity is at most~$k$, every union of vertex-disjoint cycles hits at most $k$ vertices of $S_e$.
    Applying \cref{thm:CycleDirGraphVertex} if $D_e$ is a directed graph and \cref{mainthm} if $D_e$ is bidirected to the vertex set $S_e$ in $D_e$ provides a set of at most $k$ vertices in $V(D_e)$ that hits all cycles in $D_e$ containing a vertex of $S_e$.
    This gives a set $Y_e$ of at most $k$ vertices in $V(D)$ that hits all cycles in $D$ containing an edge of $C_e$ (by possibly replacing vertices of $S_e$ by their neighbours).
    We now use these sets~$Y_e$ to describe a winning strategy for~$3k$ cops.
    
    Set~$C_0 := Y_e$ as the initial position of the cops for an arbitrary edge $e \in E(T)$.
    Note that $|C_0 | = k \leq 3k$ by assumption.
    Now suppose that after the~$i$-th round, there exists an edge~$f \in E(T)$ with~$C_i = Y_f$.
    Then $R_i \subseteq \varphi(L(T_i))$ for a component $T_i$ of~$T - f$.
    If~$T_i$ contains only a single vertex~$t$, then the~$3k$ cops catch the robber after~$i+1$ rounds by setting~$C_{i+1} := C_i \cup \{\varphi(t)\}$ which has size at most~$k + 1 \le 3k$.
    
    Otherwise, let~$f_1$ and~$f_2$ be the two edges of~$T_i$ incident with $f$.
    We set~$C_{i+1} := C_i \cup Y_{f_1} \cup Y_{f_2}$.
    Then the robber must choose a strong component~$R_{i+1}$ of~$D - C_{i+1}$ that is contained in~$\varphi(L(T_{i+1}))$ for a component $T_{i+1}$ of $T_i - f_1 - f_2$ as~$C_i \subseteq C_i \cap C_{i + 1}$.
    Let $j \in \{1,2\}$ with the property that $f_j$ is incident with $T_{i+1}$.
    Set~$C_{i+2} := Y_{f_j}$ and $T_{i+2}:= T_{i+1}$, which ensures that~$R_{i+2} \subseteq \varphi(L(T_{i+2}))$ since $C_{i+2} \subseteq C_{i+1} \cap C_{i+2}$.
    Note that~$|C_{i+1}|, |C_{i+2}| \le 3k$ and that $T_{i+2}$ is a proper subtree of $T_i$.

    As~$D$ and hence~$T$ is finite, the subtrees~$T_i$ become smaller as~$i$ increases.
    This implies that the above described strategy indeed enables $3k$ cops to catch the robber on~$D$.

    If $D$ is a directed graph, we thus have~$\dtw(D) \le 3 (3k) - 2 = 9k - 2 = 9 \cycw(D) - 2$ by~\cref{thm:CopsAndRobbers}, as desired.
    If $D$ is a circular bidirected graph, we have $\mathrm{btw}(D) \le 3 (3k) + 2 = 9k + 2 = 9 \cycw(D) + 2$ by~\cref{thm:CopsAndRobbersBidi}, as desired.
\end{proof}

\section*{Acknowledgements}

The second author carried out this work during a Humboldt Research Fellowship at the University
of Hamburg. He thanks the Alexander von Humboldt-Stiftung for financial support.

The fourth author gratefully acknowledges support by doctoral scholarships of the Studienstiftung des deutschen Volkes and the Cusanuswerk -- Bisch\"{o}fliche Studienf\"{o}rderung.

The fifth author gratefully acknowledges support by a doctoral scholarship of the Studienstiftung des deutschen Volkes.

\bibliographystyle{unsrtnat}
\bibliography{reference}

\end{document}